\documentclass [12pt,letterpaper,reqno]{amsart}

\DeclareMathSymbol{\twoheadrightarrow} {\mathrel}{AMSa}{"10}

\def\Q{{\mathbb Q}}
\def\Z{{\mathbb Z}}
\def\C{{\mathbb C}}
\def\R{{\mathbb R}}
\def\F{{\mathbb F}}

\def\Sn{{\mathbf S}_n}
\def\An{{\mathbf A}_n}

\def\Gal{\mathrm{Gal}}
\def\Alt{\mathrm{Alt}}

\def\Perm{\mathrm{Perm}}

\def\End{\mathrm{End}}

\def\Aut{\mathrm{Aut}}
\def\Hom{\mathrm{Hom}}
\def\Mat{\mathrm{Mat}}

    \def\RR{\mathfrak{R}}

\def\fchar{\mathrm{char}}

\def\dim{\mathrm{dim}}

\newtheorem{thm}{Theorem}[section]
\newtheorem{lem}[thm]{Lemma}
\newtheorem{cor}[thm]{Corollary}
\newtheorem{prop}[thm]{Proposition}
\theoremstyle{definition}
\newtheorem{defn}[thm]{Definition}
\newtheorem{ex}[thm]{Example}

\newtheorem{sect}[thm]{}

\newtheorem{rem}[thm]{Remark}

\title[Non-isogenous elliptic curves and hyperelliptic jacobians]
{Non-isogenous elliptic curves  and hyperelliptic jacobians}
\author[Yuri G. Zarhin]{Yuri G. Zarhin}
\address{Department of Mathematics, Pennsylvania State University,
University Park, PA 16802, USA}
\email{zarhin\char`\@math.psu.edu}
\dedicatory{To Yuri Ivanovich Manin on the occasion of his 85th birthday}
\thanks{The author  was partially supported by Simons Foundation Collaboration grant   \# 585711.
Part of this work was done during his stay in 2022 at the Max-Planck Institut f\"ur Mathematik (Bonn, Germany), whose hospitality and support are gratefully acknowledged.}

\begin{document}
\begin{abstract}
Let $K$ be a field of characteristic different from $2$, $\bar{K}$ its algebraic closure.
Let $n \ge 3$ be an odd prime such that $2$ is a primitive root modulo $n$.
Let $f(x)$ and $h(x)$ be degree $n$ polynomials with coefficients in $K$ and without repeated roots.
Let us consider genus $(n-1)/2$ hyperelliptic curves
$C_f: y^2=f(x)$ and $C_h: y^2=h(x)$, and their jacobians $J(C_f)$ and $J(C_h)$, which are
$(n-1)/2$-dimensional abelian varieties defined over $K$. 

Suppose that one of the
polynomials is irreducible and the other reducible.
We prove that if $J(C_f)$ and $J(C_h)$ are  isogenous over $\bar{K}$ then both jacobians are abelian varieties of CM type
with multiplication by the field of $n$th roots of $1$.

We also discuss the case when both polynomials are irreducible while their splitting fields are linearly disjoint.
In particular, we prove that if $\fchar(K)=0$, the Galois group of one of the polynomials is doubly transitive and the Galois group of the other
is a cyclic group of order $n$, then $J(C_f)$ and $J(C_h)$ are  not isogenous over $\bar{K}$.
\end{abstract}


\maketitle
\section{Definitions, notations, statements}

Let $K$ be a field,  $\bar{K}$ its algebraic closure, and
 $\Gal(K)=\Aut(\bar{K}/K)$ the group of all automorpisms  of the
field extension $\bar{K}/K$. If $K_s \subset \bar{K}$ is the separable closure of $K$ in $\bar{K}$
then $\Gal(K)$ coincides with the Galois group of the (possibly infinite) Galois extension $K_s/K$.

If $X$ is an abelian variety over $K$ then we write $\End(X)$ for the ring of its $\bar{K}$-endomorphisms and $\End^0(X)$ for the corresponding $\Q$-algebra $\End(X)\otimes\Q$. If $Y$ is (may be another) abelian variety over  $K$ then we write
$\Hom(X,Y)$ for the group of all $\bar{K}$-homomorphisms from  $X$ to
$Y$, which is a free commutative group of finite rank. It is well known that $\Hom(X,Y)=0$ if and only if
$\Hom(Y,X)=0$.

Let $f(x) \in K[x]$ be a nonconstant polynomial of degree $n$ without
repeated roots.  We write $\RR_f\subset \bar{K}$ for the  set of its
roots, and $K(\RR_f) \subset \bar{K}$ for the splitting field of $f(x)$. Then $\RR_f$ consists of $n$
elements and
$$\RR_f \subset K(\RR_f)\subset K_s\subset \bar{K}.$$
We write
$$\Gal(f)=\Gal(f/K)=\Aut(K(\RR_f)/K)=\Gal(K(\RR_f)/K)$$
 for the Galois group
of $f(x)$ over $K$.   The group $\Gal(f/K)$ permutes elements of $\RR_f$ and
therefore can be identified with a certain subgroup of the group
$\Perm(\RR_f)$ of all permutations of  $\RR_f$. (The action of $\Gal(f/K)$ on $\RR_f$ is {\sl transitive}
if and only if $f(x)$ is {\sl irreducible} over $K$.)
If we choose an order on the $n$-element set 
  $\RR_f$ then we get a group isomorphism between
$\Perm(\RR_f)$ and the full symmetric group $\Sn$, which makes
$\Gal(f/K)$ a certain subgroup of $\Sn$.  We write $\Alt(\RR_f)$ for the only index 2 subgroup of 
$\Perm(\RR_f)$, which corresponds to the alternating (sub)group $\An$ of $\Sn$ under any isomorphism between $\Perm(\RR_f)$  and $\Sn$.  Slightly abusing notation, we say that $\Gal(f)$ is $\Sn$ (resp. $\An$) if it coincides with $\Perm(\RR_f)$ (resp. $\Alt(\RR_f)$).

Throughout the paper  (unless otherwise stated)
 we assume that $\fchar(K)\ne 2$.
 
\subsection{Elliptic curves}
Let us assume that $n=3$ (i.e., $f(x)$ is a cubic polynomial) and consider the elliptic
curve
$$C_f: y^2=f(x)$$
viewed as a one-dimensional abelian variety defined over $K$ with the  infinite  point taken as the zero of group law. As usual, $j(C_f)\in K$ denotes the $j$-{\sl invariant} of the elliptic curve $C_f$.

Let $h(x)\in K[x]$ be another cubic polynomial without repeated roots and 
$$C_h: y^2=h(x)$$
be the corresponding elliptic curve (one-dimensional abelian variety) over $K$.  It is well known that $j(C_f)=j(C_h)$ if and only if the elliptic curves $C_f$ and $C_h$ are isomorphic over $\bar{K}$ \cite[Chapter III, Prop. 1.4b]{Silverman0}.

The aim of this paper is to discuss when $C_f$ and $C_h$ are {\sl not} isogenous over $\bar{K}$. 
 There are several known criteria for elliptic curves over number fields not to be isogenous that are based on their arithmetic properties (reductions, arithmetic properties of the corresponding $j$-invariants) \cite[p. 645]{MZ20}.  See also \cite[Th. 1.2]{ZarhinSh03} and \cite[Sect. 2]{ZarhinPisa}. There are also  recent results describing the ``asymptotic behavior'' of the number of elliptic curves with $j$-invariant in certain countable sets  of complex numbers (e.g., the set $\Z[\sqrt{-1}]$ of Gaussian integers) that are not isogenous to  elliptic curves whose $j$-invariants lie on a given real algebraic curve in $\C=\R^2$ \cite[Th. 1.7  and Sect. 3]{MZ20}).
 
Here we discuss  criteria that use  the properties of 
$f(x)$ and $h(x)$ over arbitrary $K$ only. 
Our main results  are the following two assertions.

\begin{thm}
\label{endo} Let $K$ be a field of characteristic different from
$2$.  Let $f(x)$ and $h(x)$ be cubic polynomials over $K$ without repeated roots.
Suppose that exactly one
of the two polynomials is irreducible.

Let us assume that $C_f$ and $C_h$ are  isogenous over $\bar{K}$. Then:

\begin{itemize}
\item[(i)]
 Both 
$\Q$-algebras
$\End^0(C_f)$ and $\End^0(C_h)$ contain a subfield isomorphic to $\Q(\sqrt{-3})$.  
\item[(ii)]
If $\fchar(K)=0$ then both $C_f$ and $C_h$ are  isogenous over $\bar{K}$ to the elliptic curve 
$y^2=x^3-1$.
\end{itemize}
\end{thm}

\begin{thm}
\label{main} Let $K$ be a field of characteristic different from
$2$. Let $f(x), h(x)\in K[x]$ be
cubic polynomials without repeated roots that enjoy the following properties.

\begin{enumerate}
\item[(i)]  The splitting fields of
 $f(x)$ and $h(x)$ are linearly disjoint over  $K$.
\item[(ii)]
$h(x)$ is irreducible over $K$.
\item[(iii)]
 $\Gal(f/K)=\mathbf{S}_3$, i.e., $K(\RR_f)$ has degree $6$ over $K$.
\end{enumerate}

If $C_f$ and $C_h$ are  isogenous over $\bar{K}$ then
$p=\fchar(K)$ is 
a prime that is not congruent to $1$ modulo $3$, and both $C_f$ and $C_h$ are supersingular elliptic curves. 
\end{thm}

\begin{rem}
If $\fchar(K)=0$ then it follows from Theorem \ref{main} that if cubic polynomials $f(x)$ and $h(x)$ enjoy properties (i), (ii), (iii) of Theorem \ref{main} 
then $C_f$ and $C_h$ are not isogenous over $\bar{K}$. This assertion is a special case (with $m=n=3$) of \cite[Th. 1.2]{ZarhinSh03}.
\end{rem}

\begin{ex}
\label{OsadaS}
 Let $K=\Q, \ f(x)=x^3-5, \ h(x)=x^3-15x+22$. Clearly,
$K(\RR_f)=\Q(\sqrt{-3}, \sqrt[3]{5})$ is a sextic number field, i.e. $\Gal(f/\Q)=\mathbf{S}_3$, and
$h(x)=(x-2)(x^2+2x-11)$ is reducible over $\Q$. So,  $f(x)$ and $h(x)$ satisfy the conditions of Theorem \ref{endo}.
It is well known that the endomorphism ring $\End(C_f)$ is $\Z\left[\frac{-1+\sqrt{-3}}{2}\right]$. It is also known \cite[Appendix A, p. 483]{Silverman} that $\End(C_h)$ is 
$$\Z+2\cdot \Z\left[\frac{-1+\sqrt{-3}}{2}\right]=\Z\left[\sqrt{-3}\right].$$

Anyway, both $C_f$ and $C_h$ are CM elliptic curves whose endomorphism algebras (over an algebraic closure $\bar{\Q}$ of $\Q$) are isomorphic to $\Q(\sqrt{-3})$. Hence, $C_f$ and $C_h$  are isogenous to each other over
$\bar{\Q}$ and to $y^2=x^3-1$. (Actually, $C_f$ is even isomorphic to $y^2=x^3-1$ over $\bar{\Q}$).
\end{ex}

\begin{rem}
\label{Chow}
Let $\tilde{K}$ be an overfield of $K_s$. If $X$ and $Y$ are abelian varieties over $K$ then, by a theorem of Chow (\cite[Ch. 2, Th. 5]{LangAV}, \cite[Th. 3.19]{Conrad}), all their $\tilde{K}$-homomorphisms (and $\tilde{K}$-endomorphisms) are defined over $K_s$. In particular, $X$ and $Y$ are isogenous over $\tilde{K}$ if and only if they are isogenous over $K$.
\end{rem}

\begin{cor}
\label{rational}
Let $K=\Q$.
Let $f(x)\in \Q[x]$ be an irreducible cubic polynomial and $h(x)\in \Q[x]$ a reducible cubic polynomial. 
 
Then precisely one of the following two conditions holds.

\begin{itemize}
\item[(i)]
The elliptic curves $C_f$ and $C_h$ are not isogenous over $\bar{\Q}$ (and even over $\C$).
\item[(ii)]
Both $j(C_f)$ and $j(C_h)$ lie in a $3$-element set
$$S=\{0,\ 2^4 3^3 5^3, \ -2^{15}3\cdot 5^3\}\subset\Q.$$
\end{itemize}
\end{cor}
\begin{proof}[Proof of Corollary \ref{rational} (modulo Theorem \ref{endo})]
Let us consider an elliptic curve $C$  that is defined over $\Q$. Then $\End^0(C)$ contains $\Q(\sqrt{-3})$ (actually coincides with it) 
if and only if the $j$-invariant of $C$ lies in $S$ (\cite[Sect. 2, p. 295]{SerreCM}, \cite[Appendix A, Sect. 3]{Silverman}). Now the desired result follows from Theorem \ref{endo}
combined with Remark \ref{Chow}.
\end{proof}

\begin{cor}
\label{S3function}
Let $k$ be an algebraically closed field of characteristic $0$ and $K$ an overfield of $k$.
Let $f(x)\in K[x]$ be an irreducible cubic polynomial with $\Gal(f/K)=\mathbf{S}_3$. (E.g., $K=k(t)$ is the field of rational functions in one variable $t$ over $k$
 and $f(x)=x^3-x-t$). 
 Let $h(x)\in K[x]$ be a reducible cubic polynomial
without repeated roots.

Then the elliptic curves $C_f$ and $C_h$ are not isogenous over $\bar{K}$.
\end{cor}

\begin{proof}[Proof of Corollary \ref{S3function} (modulo Theorem \ref{endo})]
Clearly,  the field $k$ contains $\bar{\Q}$. By Lemma 2.4 on p. 366 of \cite{ZarhinPisa}, $j(C_f) \not \in k$.  In particular, $j(C_f) \not \in \bar{\Q}$. 
In light of \cite[Ch. II, Sect. 6, Th. 6.1]{Silverman},
$C_f$ is {\sl not} of CM type over $\bar{K}$, i.e.,
$$\End(C_f)=\Z, \ \End^0(C_f)=\Q.$$
Now the desired result follows from Theorem \ref{endo}.
\end{proof}

We will need the following elementary observation that will be proven in Section \ref{crucialP}.

\begin{prop}
\label{BourbakiLang}
Let $n\ge 3$ be a prime, $K$ a field.  Let $f(x), h(x) \in K[x]$
be degree $n$ irreducible polynomials without repeated roots. Suppose that  $\Gal(h/K)$
is a cyclic group of order $n$, i.e., the Galois extension $K(\RR_h)/K$ is  cyclic  of degree $n$.
Then precisely one of the following two conditions holds.

\begin{itemize}
\item[(i)]
The fields $K(\RR_f)$ and $K(\RR_h)$ are linearly disjoint over $K$.
\item[(ii)]
The field $K(\RR_f)$ contains $K(\RR_h)$ and the cyclic group $\Z/n\Z$ of order $n$
is isomorphic to a quotient of $\Gal(f/K)$. In addition, the permutation group
$\Gal(f/K)$ is not doubly transitive.
\end{itemize}
\end{prop}

\begin{cor}
\label{S3A3}
Let $K$ be a field of characteristic different from
$2$. Let $f(x), h(x)\in K[x]$ be
irreducible cubic polynomials without repeated roots such that 
$$\Gal(f/K)=\mathbf{S}_3, \ \Gal(h/K)=\mathbf{A}_3.$$

If $C_f$ and $C_h$ are  isogenous over $\bar{K}$ then
$p=\fchar(K)$ is a prime that is congruent to $2$ modulo $3$, and both $C_f$ and $C_h$ are supersingular elliptic curves. 
\end{cor}

\begin{proof}[Proof of Corollary \ref{S3A3} (modulo Theorem \ref{main})]
Since  the order of the group $\mathbf{A}_3$ is $3$, the degree
$[K(\RR_h):K]=3$. 
Notice that $n=3$ is a prime, the permutation  group $\Gal(\RR_f)=\mathbf{S}_3$ is {\sl doubly transitive} and $\Gal(h)=\mathbf{A}_3\cong \Z/3\Z$.
Applying Proposition \ref{BourbakiLang}, we conclude that $K(\RR_f)$ and $K(\RR_h)$ are linearly disjoint over $K$. Now the desired result follows from Theorem \ref{main}.
\end{proof}

\begin{cor}
\label{overQ}
Suppose that $K=\Q$ and $h(x)\in \Q[x]$ is an irreducible cubic polynomial such that $\Q(\RR_h)$ is a cyclic cubic field,
i.e., $\Gal(h/\Q)=\mathbf{A}_3$.   Then the elliptic curve $C_h$ enjoys the following properties.

\begin{itemize}
\item[(i)]
$C_h$ is not isogenous  to $y^2=x^3-1$ over $\bar{\Q}$, i.e., $\End^0(C_h)$ is not isomorphic to $\Q(\sqrt{-3})$.
In other words,
$$j(C_h)\ne 0,\ 2^4 3^3 5^3, \ -2^{15}3\cdot 5^3.$$
\item[(ii)]
Let $u(x)$ be an an irreducible cubic polynomial such that $K(\RR_u)$ is a cyclic cubic field,
i.e., $\Gal(u/\Q)=\mathbf{A}_3$.   If the cubic fields  $\Q(\RR_h)$ and $\Q(\RR_u)$  are not isomorphic then
the elliptic curves $C_h$ and $C_u$ are not isogenous over $\bar{\Q}$.
\end{itemize}
\end{cor}

\begin{proof}[Proof of Corollary \ref{overQ} (modulo Theorems \ref{main} and \ref{endo})]
In order to prove (i), 
let us put $$f(x)=x^3-5\in \Q[x].$$ We have seen (Example \ref{OsadaS}) that $\Gal(f/\Q)=\mathbf{S}_3$.
It follows from Corollary \ref{S3A3} that the elliptic curves $C_f:y^2=x^3-5$ and $C_{h}$
are {\sl not} isogenous over $\bar{\Q}$ and even over $\C$, thanks to Remark \ref{Chow}. 
Since $\End^0(C_f) \cong \Q(\sqrt{-3})$, it follows from \cite[Ch. 4, Sect. 4.4, Prop. 4.9]{ShimuraIA} combined with  Remark \ref{Chow}. 
that $\End^0(C_h)$ is {\sl not} isomorphic to $\Q(\sqrt{-3})$.  On the other hand, it is well known (see the proof of Corollary  \ref{rational})
that if $C$ is an elliptic curve over $\Q$ then 
 $$j(C) \in \{0,\ 2^4 3^3 5^3, \ -2^{15}3\cdot 5^3\}$$ if and only if $\End^0(C)$ is isomorphic to $\Q(\sqrt{-3})$.
This ends the proof of  (i).

In order to prove (ii), notice that $\Q(\RR_h) \ne \Q(\RR_u)$ (they both are subfields of $\bar{\Q}$). Since they both have the same degree over $\Q$ (namely, $3$, which is a prime),
$\Q(\RR_h)$ does {\sl not} contain $\Q(\RR_u)$. Applying Proposition \ref{BourbakiLang}, we conclude that 
the fields  $\Q(\RR_h)$ and $\Q(\RR_u)$  
 are {\sl linearly disjoint} over $\Q$. The linear disjointness implies 
 that  $u(x)$ remains irreducible  over
 $K_1=\Q(\RR_h)$ while $h(x)$  is reducible (actually splits into a product of linear factors) over $K_1$. Notice that $\bar{\Q}$ is an algebraic closure of $K_1$.
 
 Applying Theorem \ref{endo}  to irreducible $u(x)$ and reducible $h(x)$ over $K_1$, we conclude that if $C_{h}$ and $C_{u}$ are isogenous over $\bar{\Q}$ then $C_{h}$ is isogenous over $\bar{\Q}$  to $y^2=x^3-1$, which is not the case, in light of already proven (i). Hence, 
 $C_{h}$ and $C_{u}$ are {\sl not} isogenous over $\bar{\Q}$.
\end{proof}

\begin{ex}
\label{simpleCubic}
\begin{itemize}
\item[(i)]
Let  us put
$$K=\Q, \ a  \in \Z, \  \ h_a(z):=x^3-ax^2-(a+3)x-1 \in \Q[x].$$
It is known \cite[p. 1137--1138]{Shanks} that  for every integer $a$ the splitting field  $\Q(\RR_{h_a})$ of the cubic polynomial 
$h_a(z)$
 is a {\sl cyclic} cubic field, i.e. $\Gal(h_a/\Q) =\mathbf{A}_3$. 
 If $f(x) \in \Q[x]$ is any irreducible cubic polynomial with $\Gal(f)=\mathbf{S}_3$ (e.g., $f(x)=x^3-x-1$ or $x^3-5$) then it follows from Corollary \ref{S3A3} that the elliptic curves $C_f$ and $C_{h_a}: y^2=h_a(x)$ are {\sl not} isogenous 
  over $\bar{\Q}$ (and even over $\C$) for all $a \in \Z$.  
  \item[(ii)]
  Suppose that  $a \ge -1$ and $a^2+3a+9$ is a {\sl prime}, e.g.,
  $$a=-1, 1, 2,  4, 7, 8, 10, 11, 16, 17,  \dots,    410$$ \cite[Table 1 on p. 1140]{Shanks}.  Then the discriminant of $\Q(\RR_{h_a})$ is $(a^2+3a+9)^2$ \cite[p. 1138]{Shanks}.  This implies that if $b$ is an integer such that $b > a$
   and $b^2+3b+9$ is also a prime then 
    $(b^2+3b+9)^2$ is the discriminant of  $\Q(\RR_{h_b})$ and
    $$(a^2+3a+9)^2 < (b^2+3b+9)^2.$$
    Hence, the cubic fields $\Q(\RR_{h_a})$ and 
 $\Q(\RR_{h_b})$ have {\sl distinct} discriminants and therefore are {\sl not} isomorphic.  
 In light of Corollary \ref{overQ}, the elliptic curves $C_{h_a}$ and $C_{h_b}$ are  {\sl not} isogenous over $\bar{\Q}$.
  \end{itemize}
\end{ex}
We deduce Theorems \ref{endo} and \ref{main} from  more general results about non-isogenous hyperelliptic jacobians 
(Theorems \ref{endoH} and \ref{mainH} below)
that will be stated in Subsection \ref{hyperJ}
and proven in Section \ref{mainproof}.

\subsection{Non-isogenous hyperelliptic jacobians}
\label{hyperJ}

Throughout this subsection, $n\ge 3$ is an odd integer, $f(x)$ and $h(x)$ are degree $n$ polynomials with coefficients in $K$ and without repeated roots,
$$C_f: y^2=f(x), \ C_h: y^2=h(x)$$
are the corresponding genus $(n-1)/2$ hyperelliptic curves over $K$, whose jacobians we denote by $J(C_f)$ and $J(C_h)$, respectively. These jacobians are $(n-1)/2$-dimensional abelian varieties defined over $K$.

\begin{thm}
\label{endoH}
Suppose that $n\ge 3$ is an odd prime such that $2$ is a primitive root $\bmod\ n$.
Let $K$ be a field of characteristic different from
$2$. 
Let $f(x), h(x) \in K[x]$ be degree $n$ polynomials without repeated roots.  Suppose that one of the
polynomials is irreducible and the other reducible.

If the abelian varieties $J(C_f)$ and $J(C_h)$ are  isogenous over $\bar{K}$ then both jacobians
are abelian varieties of CM type over $\bar{K}$ with multiplication by  the $n$th cyclotomic field $\Q(\zeta_n)$.
\end{thm}

\begin{rem}
See \cite[Th. 1.1]{Pip} where the possible structure of the endomorphism algebra $\End^0(J(C_f))$ is described when $K$ is a number field,
$f(x)$ is a prime (odd) degree $n$ irreducible polynomial over $K$ and $2$ is a primitive root $\bmod \ n$.
(See also \cite{ZarhinMRL,ZarhinTexel}.)
\end{rem}

\begin{thm}
\label{mainH}
Suppose that $n\ge 3$ is an odd prime such that $2$ is a primitive root $\bmod \ n$.
Let $K$ be a field of characteristic different from
$2$. 
Suppose that $f(x), h(x) \in K[x]$ are degree $n$ polynomials without repeated roots that enjoy the following properties.

\begin{itemize}
\item[(i)]
$f(x)$ is irreducible over $K$ and its Galois group $\Gal(f) \subset \Perm(\RR_f)$ is doubly transitive.
\item[(ii)]
$h(x)$ is irreducible over $K$.
\item[(iii)]
The splitting fields $K(\RR_f)$ and $K(\RR_h)$ of $f(x)$ and $h(x)$ are linearly disjoint over $K$.
\end{itemize}
Then the hyperelliptic jacobians $J(C_f)$ and $J(C_h)$ enjoy precisely one of the following properties.
\begin{enumerate}
\item[(1)]
The abelian varieties $J(C_f)$ and $J(C_h)$ are not isogenous over $\bar{K}$. 
Even better,
$$\Hom(J(C_f),J(C_h))=\{0\}, \ \Hom(J(C_h),J(C_f))=\{0\}.$$
\item[(2)]
$p=\fchar(K)>0$ and  both $J(C_f)$ and $J(C_h)$ are  supersingular abelian varieties.
In addition,  $n$ does not divide $p-1$.
More precisely, if  $p \ne n$ and $\mathfrak{f}_p$ is the order of
$p \bmod  n$ in the multiplicative group $(\Z/n\Z)^{*}$ then $\mathfrak{f}_p$ is even.

\end{enumerate}
\end{thm}

\begin{rem}
If $n=3$ then $C_f$ and $C_h$ are elliptic curves that are canonically isomorphic to their jacobians. Clearly, $\mathbf{S}_3$ is a {\sl doubly transitive} permutation group while $2$ is a {\sl primitive root} $\bmod \ 3$. Hence, Theorem \ref{endo}(i) and Theorem  \ref{main} are special cases of Theorems \ref{endoH} and \ref{mainH}, respectively.
Now Theorem \ref{endo}(ii) follows from Theorem \ref{endo}(i)  combined with \cite[Ch. 4, Sect. 4.4, Prop. 4.9]{ShimuraIA}  and Remark \ref{Chow}.
\end{rem}

\begin{ex}
\label{SelmerG}
 Let $K=\Q$ and $n \ge 3$ is an odd integer. Let $$f(x)=x^n-x-1, h(x)=x^n-1.$$  By a theorem of Osada \cite[Cor. 3 on p. 233]{Osada},
$\Gal(f)=\Sn$, which is doubly transitive. (The irreducibility of $f(x)$ was proven by Selmer \cite[Th. 1]{Selmer}.) On the other hand, $h(x)$ is obviously reducible over $\Q$. It is well known that $J(C_h)$ is an abelian variety of CM type with multiplication by the $n$th cyclotomic field $\Q(\zeta_n)$. It was proven in \cite[Examples 2.2]{ZarhinMRL} that $J(C_f)$ is absolutely simple (and even $\End(J(C_f))=\Z$). Hence, $J(C_f)$ and $J(C_h)$ are not isogenous over $\bar{\Q}$ (and even over $\C$), and 
$$\Hom(J(C_f),J(C_h))=\{0\}, \ \Hom(J(C_h),J(C_f))=\{0\}.$$
\end{ex}

\begin{ex} Let $K=\Q$ and $n \ge 3$ is an odd integer. Let  $$f(x)=x^n-2, \ h(x)=x^n-1.$$ By Eisenstein's criterion, the polynomial $f(x)$ is irreducible over the field $\Q_2$
of $2$-adic numbers
and therefore is irreducible also over $\Q$
 (see also \cite[Ch. VI, Sect. 9, Th. 9.1]{Lang}). Obviously, $h(x)$ is reducible over $\Q$. However, $J(C_f)$ and $J(C_h)$ are not only isogenous over $\bar{\Q}$ but actually become isomorphic over $\bar{\Q}$. On the other hand, both   $J(C_f)$ and $J(C_h)$ are abelian varieties of CM type with multiplication by $\Q(\zeta_n)$.
\end{ex}

The following assertion may be viewed as a generalization of Corollary \ref{S3A3}.

\begin{cor}
\label{S3A3G}
Suppose that $n\ge 3$ is an odd prime such that $2$ is a primitive root $\bmod \ n$.
Let $K$ be a field of characteristic different from
$2$. 
Suppose that $f(x), h(x) \in K[x]$ are degree $n$ polynomials without repeated roots that enjoy the following properties.

\begin{itemize}
\item[(i)]
$f(x)$ is irreducible over $K$ and its Galois group $\Gal(f) \subset \Perm(\RR_f)$ is doubly transitive.
\item[(ii)]
$h(x)$ is irreducible over $K$ and $K(\RR_h)/K$ is a degree $n$ cyclic field extension, i.e., $\Gal(h) \cong \Z/n\Z$.
\end{itemize}
Then the hyperelliptic jacobians $J(C_f)$ and $J(C_h)$ enjoy precisely one of the following properties.
\begin{enumerate}
\item[(1)]
$\Hom(J(C_f),J(C_h))=\{0\}, \ \Hom(J(C_h),J(C_f))=\{0\}$.
\item[(2)]
$p=\fchar(K)>0$ and  both $J(C_f)$ and $J(C_h)$ are  supersingular abelian varieties. 
In addition, $n$ does not divide $p-1$.
More precisely, if
$p \ne n$ then the residue
$p \bmod  n$ has even multiplicative order in $(\Z/n\Z)^{*}$.
\end{enumerate}
\end{cor}

\begin{proof}[Proof of Corollary \ref{S3A3G} (modulo Theorem \ref{mainH})]
Applying Proposition \ref{BourbakiLang}, we conclude that the fields
$K(\RR_f)$ and $K(\RR_h)$ are linearly disjoint over $K$. Now the desired result follows from Theorem \ref{mainH}.
\end{proof}

\begin{ex}
Let us put $K=\Q, n=5$, and
$$ f_1(x)=x^5-x-1,  \ f_2(x)=x^5+15x+12\in \Q[x];$$
$$h(x)=x^5-110x^3-55x^2+2310x+979\in \Q[x].$$
Notice that $2$ is a {\sl primitive root} $\bmod \ 5$.
We have seen (Example \ref{SelmerG}) that $f_1(x)$ is irreducible with the doubly transitive Galois group $\mathbf{S}_5$. It is known \cite[Sect. 5, p. 398]{Dummit}
that $f_2(x)$ is irreducible, whose Galois group  is the {\sl doubly transitive  Frobenius group} $\mathbf{F}_{20}$ of order $20$ \cite[p. 388]{Dummit}. 
It is also known  \cite[p. 400]{Dummit} that $h(x)$ is irreducible, whose Galois group is a  cyclic group of order $5$.
Applying Corollary \ref{S3A3G} to the pair of polynomials $f_1(x)$ and $h(x)$, and to the pair of polynomials  $f_2(x)$ and $h(x)$, we conclude that
$$\Hom(J(C_{f_1}),J(C_h))=\{0\}, \ \Hom(J(C_h),J(C_{f_1}))=\{0\};$$
$$\Hom(J(C_{f_2}),J(C_h))=\{0\}, \ \Hom(J(C_h),J(C_{f_2}))=\{0\}.$$

I claim that
$$\Hom(J(C_{f_1}),J(C_{f_2}))=\{0\}, \ \Hom(J(C_{f_2}),J(C_{f_1}))=\{0\}.$$

Indeed, in light of Theorem \ref{mainH}, it suffices to check that the splitting fields $\Q(\RR_{f_1})$ and $\Q(\RR_{f_2})$
are {\sl linearly disjoint} over $\Q$, i.e., their {\sl intersection} $F=\Q(\RR_{f_1})\cap\Q(\RR_{f_2})$ coincides with $\Q$.
Suppose that this is not the case, i.e. $[F:\Q]>1$. Clearly,  $F/\Q$ is a Galois extension. It is also clear that $[\Q(\RR_{f_1}):\Q]>[\Q(\RR_{f_2}):\Q]$ and the Frobenius group $\Gal(f_2)$
is {\sl not} isomorphic to a quotient of $\Gal(f_1)=\mathbf{S}_5$.  Hence,  none of $\Q(\RR_{f_1})$ and  $\Q(\RR_{f_2})$ is  a subfield of the other one,
and therefore $F$ is a {\sl proper subfield} of both $\Q(\RR_{f_1})$ and $\Q(\RR_{f_2})$. In particular, the natural {\sl surjective} homomorphism
$$\mathbf{S}_5=\Gal(f_1)\twoheadrightarrow \Gal(F/\Q)$$
has a {\sl nontrivial} kernel and therefore $\Gal(F/\Q)$ is a cyclic group of order $2$, i.e., $F$ is a quadratic field. Since $F$ is a subfield of 
$\Q(\RR_{f_1})$ and the only index $2$ subgroup of $\mathbf{S}_5$ is $\mathbf{A}_5$,  the field $F=\Q(\sqrt{D_1})$ where $D_1$ is the discriminant of $f_1(x)$,
which  equals $19\cdot 151$. On the other hand, $\Q(\RR_{f_2})$ contains the quadratic subfield $\Q(\sqrt{D_2})$ where $D_2$ is the discriminant of $f_1(x)$,
which  equals $2^{10}3^4 5^5$  \cite{Dummit} and therefore  $\Q(\RR_{f_2})$ contains the quadratic field $F_2:=\Q(\sqrt{5})$. This implies that $\Q(\RR_{f_2})$ contains
two {\sl distinct} quadratic fields $F$ and $F_2$ and therefore there is a surjective group homomorphism
$$\Gal(\Q(\RR_{f_2})/\Q)\twoheadrightarrow \Gal(F_2F/\Q) \cong \Z/2\Z \times \Z/2\Z.$$

However,  the Sylow $2$-subgroup of the Frobenius group $\Gal(\Q(\RR_{f_2})/\Q)$ is {cyclic}. Hence, $\Gal(\Q(\RR_{f_2})/\Q)$ does not have a quotient that is
isomorphic to $\Z/2\Z \times \Z/2\Z$. The obtained contradiction proves the desired result.

\end{ex}

The paper is organized as follows. 
In Section \ref{order2} we discuss Galois properties of points of order 2 on hyperelliptic jacobians
and abelian varieties. Notice that Proposition \ref{isogEll}  is central to the  results of the paper and may be of certain independent interest.
 We also formulate there several useful auxiliary results about matrix algebras of skew-fields and splitting fields of polynomials.  
Section \ref{mainproof}  contains the proofs of  Theorems \ref{endoH} and \ref{mainH}.
In Sections \ref{crucialP} and \ref{cyclicProof} we prove Proposition \ref{BourbakiLang} and auxiliary results from Section  \ref{order2}.

This paper may be viewed as a follow-up of  \cite{ZarhinSh03,ZarhinMZ06}. 

{\bf Acknowledgements}. I am grateful to Alexei Skorobogatov and Umberto Zannier for their interest in this topic.
My special thanks go to Ken Ribet and  both referees for useful comments that helped to improve the exposition and simplify the arguments.

\section{Order $2$ points on hyperelliptic jacobians}
\label{order2}
Recall that $ \fchar(K)\ne 2$.
We start with an arbitrary positive-dimensional abelian variety  $X$  over $K$. If 
$d$ is a positive integer that is not divisible
by $\fchar(K)$ then we write $X[d]$ for the kernel of
multiplication by $d$ in $X(\bar{K})$.  It is well known that
$$X[d]\subset X(K_s)\subset X(\bar{K});$$
in addition, $X[d]$ is a $\Gal(K)$-submodule of $X(K_s)$; this submodule is isomorphic as a commutative group to $(\Z/d\Z)^{2\dim(X)}$ 
(\cite[Sect. 6]{Mumford}, \cite[Sect. 8, Remark 8.4]{Milne}).
We denote by
$$\tilde{\rho}_{d,X}:\Gal(K) \to \Aut_{\Z/d\Z}(X[d])$$
 the corresponding (continuous)
homomorphism defining the  action of $\Gal(K)$ on $X[d]$ and  put
$$\tilde{G}_{d,X}:=\tilde{\rho}_{d,X}(\Gal(K)) \subset
\Aut_{\Z/d\Z}(X[d]).$$ 
Let us consider
$$G(d):=\ker(\tilde{\rho}_{d,X})\subset \Gal(K),$$
which is a closed {\sl normal subgroup} of finite index (and therefore also open) in $\Gal(K)$. 
Since $G(d)$ is open  normal, the subfield of $G(d)$-invariants in $K_s$
$$K(X[d]):=K_s^{G(d)}$$
is a finite Galois extension of $K$. Hence,   $\tilde{G}_{d,X}=\Gal(K)/G(d)$ coincides with
the Galois group of   $K(X[d])/K$. 
By definition,   $K(X[d])$ coincides with the {\sl field of definition} of all torsion points of order dividing $d$ on
$X$.  
For example, 
$X[2]$ is a $2\dim(X)$-dimensional vector space over the $2$ elements
field $\F_2=\Z/2\Z$ and the inclusion $\tilde{G}_{2,X}
\subset \Aut_{\F_{2}}(X[2])$ defines a {\sl faithful} linear
representation of the group  $\tilde{G}_{2,X}$ in the vector
space $X[2]$. 


\begin{rem}
\label{tildeG}
The surjectiveness of $\Gal(K) \twoheadrightarrow \tilde{G}_{2,X}$  implies that
 the $\Gal(K)$-module $X[2]$ is simple (resp. absolutely simple) if and only if the $\tilde{G}_{2,X}$-module $X[2]$ is absolutely simple.
\end{rem}

\begin{rem}
\label{silver}
\begin{itemize}
\item[(i)]
Let $K(X[4])$ be the field of definition of all points of order $4$ on $X$. 
Recall that $K(X[4])/K$ is a finite Galois field extension. Clearly, 
$$K\subset K(X[2])\subset K(X[4])\subset K_s.$$

It is known that  the Galois group $\Gal(K(X[4])/K(X[2]))$ is a finite commutative group of exponent $2$ or $1$. For reader's convenience, let us give a short proof.
Let 
$$\sigma \in \Gal(K(X[4])/K(X[2]))\subset \Gal(K(X[4])/K)=\tilde{G}_{4,X}.$$
Then for each $x \in X[4]$ we have
$2x \in X[2]$ and therefore $\sigma(2x)=2x$. This implies that
$2(\sigma(x)-x)=\sigma(2x)-2x=0$, i.e., $y=\sigma(x)-x\in X[2]$ and therefore $\sigma(y)=y$.
Thus
$$\sigma^2(x)=\sigma(x+y)=\sigma(x)+y=(x+y)+y=x+2y=x.$$



This proves that each $\sigma \in \Gal(K(X[4])/K(X[2]))$ has order dividing $2$ and therefore
$\Gal(K(X[4])/K(X[2]))$ is a (finite) commutative group of exponent $2$ or $1$.
\item[(ii)]
All the endomorphisms of $X$ are defined over $K(X[4])$ (a special case of Theorem  2.4 of A. Silverberg \cite{Silverberg}). See \cite{GK,Remond,Pip} for further results about the field of definition of endomorphisms of abelian varieties.
\end{itemize}
\end{rem}

The following two assertions will be used in the proof of Theorems \ref{endoH} and \ref{mainH}.

\begin{lem}
\label{simpleX2}
\begin{footnote}
{Actually, the statement and proof of Lemma \ref{simpleX2} below
may be extracted from  \cite[p. 4644--4645]{Pip}.}
\end{footnote}
Let $n$ be an odd prime  such that $2$ is a primitive root $\bmod \ n$ of $1$. Let $g=(n-1)/2$.
Suppose that $\dim(X)=g$ and
the degree $[K(X[2]):K]$ is divisible by  $n$.

Then the $\Gal(K)$-module $X[2]$ is simple.
\end{lem}

\begin{prop}
\label{isogEll}
Let $n$ be an odd prime  such that $2$ is a primitive root $\bmod \ n$. Let $g=(n-1)/2$.
Let $X$ and $Y$ be $g$-dimensional abelian varieties defined over $K$ that are isogenous over $\bar{K}$. Suppose that $K(Y[2])=K$ and the degree
 $[K(X[2]):K]$ is divisible by  $n$.

Then both $X$ and $Y$ are abelian varieties of CM type over $\bar{K}$ with multiplication by $\Q(\zeta_{n})$.

\end{prop}

We will prove Lemma \ref{simpleX2} and Proposition \ref{isogEll} in Section \ref{crucialP}.

Now let us turn to hyperelliptic jacobians. The following assertion will play a crucial role in the proof of Theorems \ref{endoH} and \ref{mainH}.

\begin{prop}
\label{crucial}
Suppose that $n\ge 3$ is an odd prime and $2$ is a primitive root $\bmod \ n$.
Let $f(x)\in K[x]$  be a degree $n$  polynomial without repeated roots.

Then the jacobian $J(C_f)$ enjoys the following properties.

\begin{enumerate}
\item[(0)]
There is a canonical  isomorphism of finite groups
$$\tilde{G}_{2,J(C_f)} \cong \Gal(f/K).$$
In addition, $K(\RR_f)$ coincides with the field $K(J(C_f)[2])$ of definition of all points of order $2$ on $J(C_f)$.
\item[(1)]
If $f(x)$ is irreducible over $K$ then $\Gal(f)$ contains a cyclic subgroup $H$ of order $n$ and the $\Gal(K)$-module $J(C_f)[2]$ is simple.
\item[(2)]
If $f(x)$ is irreducible over $K$ and its Galois group $\Gal(\RR_f) \subset \Perm(\RR_f)$ is doubly transitive then the $\Gal(K)$-module $J(C_f)[2]$ is absolutely simple.
\item[(3)]
Suppose that $f(x)$ is irreducible over $K$ and its Galois group $\Gal(\RR_f) \subset \Perm(\RR_f)$ is doubly transitive. 
Let $h(x)\in K[x]$  be a degree $n$  polynomial without repeated roots that is irreducible over $K$.

If the splitting fields $K(\RR_f)$ and $K(\RR_h)$ of $f(x)$ and $h(x)$ are linearly disjoint over $K$
then either
$$\Hom(J(C_f),J(C_h))=\{0\}, \ \Hom(J(C_h),J(C_f))=\{0\}$$
or $\fchar(K)>0$ and both $J(C_f)$ and $J(C_h)$ are supersingular abelian varieties.

\end{enumerate}
\end{prop}

We prove Proposition \ref{crucial} in Section \ref{crucialP}.

\subsection{Galois module $J(C_f)[2]$}
\label{QRR}
In  this subsection we discuss  a  well known  explicit description of the Galois module $J(C_f)[2]$
for arbitrary (separable) $f(x)$ and (odd) $n$.
This description will be used in the proof of Proposition \ref{crucial}.  Let us start with the $n$-dimensional $\F_2$-vector space
$$\F_2^{\RR_f}=\{\phi:\RR_f \to \F_2\}$$
of all $\F_2$-valued functions on $\RR_f\subset K_s$. The  action of $\Perm(\RR_f)$ on $\RR_f$ provides $\F_2^{\RR_f}$ with the structure of  faithful $\Perm(\RR_f)$-module, which splits into a direct sum
$$\F_2^{\RR_f}=\F_2\cdot {\bf 1}_{\RR_f}\oplus Q_{\RR_f}$$
of the one-dimensional subspace of constant functions $\F_2\cdot {\bf 1}_{\RR_f}$ and the $(n-1)$-dimensional {\sl heart} \cite{Klemm,Mortimer}
$$Q_{\RR_f}=\{\phi:\RR_f \to \F_2\mid \sum_{\alpha\in\RR_f}\phi(\alpha)=0\}$$
(here we use that $n$ is odd). Clearly, the  $\Perm(\RR_f)$-module is faithful. It remains faithful if we view it as the $\Gal(f)$-module.  

The field inclusion $K(\RR_f)\subset K_s$ induces the {\sl surjective} continuous homomorphism
$$\Gal(K)=\Gal(K_s/K)\twoheadrightarrow \Gal(K(\RR_f)/K)=\Gal(f),$$
which gives rise to the natural structure of the $\Gal(K)$-module on $Q_{\RR_f}$ such that the image of $\Gal(K)$ in $\Aut_{\F_2}(Q_{\RR_f})$ coincides with
$$\Gal(f)\subset \Perm(\RR_f)\hookrightarrow \Aut_{\F_2}(Q_{\RR_f}).$$ 
The surjectiveness implies that the $\Gal(f)$-module $Q_{\RR_f}$ is simple (resp. absolutely simple) if and only if it 
is simple (resp. absolutely simple) as the $\Gal(K)$-module.

It is well known (see, e.g., \cite{Mori2,ZarhinTexel}) that the $\Gal(K)$-module $J(C_f)[2]$ and $Q_{\RR_f}$ are canonically  isomorphic. This implies that the groups $\tilde{G}_{2,J(C_f)}$ and $\Gal(f)$ are canonically isomorphic. It is also clear that $K(\RR_f)$ coincides with  $K(J(C_f)[2])$.
In addition, the $\Gal(K)$-module $J(C_f)[2]$ is simple (resp. absolutely simple) if and only if the $\Gal(f)$-module $Q_{\RR_f}$ is simple (resp. absolutely simple).

\subsection{Useful algebraic results}

We finish this section by stating two auxiliary elementary results that will be  used in the proofs of Theorems \ref{endoH} and \ref{mainH}.

\begin{lem}
\label{degSplit}
Let $n$ be a prime, $K$ an arbitrary  field, and $h(x)\in K[x]$ a  polynomial without repeated roots with $\deg(h)\le n$. If $h(x)$ is reducible over $K$ then the degree $[K(\RR_h):K]$ is not divisible by $n$. More precisely, if $l$ is a prime divisor of $[K(\RR_h):K]$ then $l<n$.
\end{lem}

\begin{lem}
\label{divisionA}

Let $n$ and $p$ be distinct  primes. Assume that $n$ is odd,  i.e., $m:=(n-1)/2$ a positive integer. Let $D$ be a quaternion  algebra over $\Q$ that is ramified at $p$, i.e., $D\otimes_{\Q}\Q_p$ is a  division algebra over the field $\Q_p$ of $p$-adic numbers. Let $E$ be a commutative $\Q$-subalgebra of $\mathrm{Mat}_m(D)$ (with the same $1$) that is isomorphic to the $n$th cyclotomic field $\Q(\zeta_{n})$.

Then $p \bmod n$ has even multiplicative order in $(\Z/n\Z)^{*}$. In particular, $n$ does not divide $p-1$.
\end{lem}

\begin{rem}
Keeping the notation and assumptions  of Lemma \ref{divisionA}, assume additionally that  $n=3$. (This is exactly the case that we need for elliptic curves.)
Then 
$$p \ne 3, \ \Q(\zeta_3)=\Q(\sqrt{-3}), \ m=1,  \mathrm{Mat}_m(D)=D.$$

 Hence,
$\Q(\sqrt{-3})\otimes_{\Q}\Q_p$ is isomorphic to a subalgebra of the division algebra $D\otimes_{\Q}\Q_p$. This implies that $\Q(\zeta_3)\otimes_{\Q}\Q_p$ has no zero divisors, i.e., $p$ is {\sl inert} in $\Q(\sqrt{-3})$, which means that $p-1$ is {\sl not} divisible by $3$. Hence, $p \equiv 2 \bmod 3$  and therefore has multiplicative order $2$ in $(\Z/3\Z)^{*}$.
This proves Lemma \ref{divisionA} for $n=3$.
\end{rem}

We will prove Lemmas \ref{degSplit} and \ref{divisionA} in Sections \ref{crucialP} and \ref{cyclicProof} respectively.

\section{Non-isogenous hyperelliptic jacobians: proofs}
\label{mainproof} 

\begin{proof}[Proof of Theorem \ref{endoH}]
We may assume that $f(x)$ is irreducible over $K$ and $h(x)$ is reducible over $K$.

By Proposition \ref{crucial}(1), 
$$\Gal(K(J(C_f)[2]))/K)=\Gal(K(\RR_f)/K)$$ 
contains a cyclic subgroup $H$ of order $n$.
Replacing $K$ by its overfield $K(\RR_f)^H$ (of $H$-invariants in $K(\RR_f)$), we may and will assume that 
$$\Gal(K(J(C_f)[2])/K)=\Gal(K(\RR_f)/K)=H$$ is a cyclic group of (prime odd) order $n$.  By Lemma \ref{degSplit}, the degree of Galois extension
$K(\RR_h)/K$ is not divisible by $n$. This implies that the fields $K(J(C_f)[2])$ and $K(\RR_h)$ are linearly disjoint over $K$. Replacing $K$ by  its overfield $K(\RR_h)$,
we may and will assume that
$$K(J(C_h)[2])=K(\RR_h)=K$$
and $\Gal(K(J(C_f)[2])/K)$ is a cyclic group of order $n$. In particular, 
$$[K(J(C_f)[2]):K]=n.$$  Now the desired result follows from Proposition
\ref{isogEll} applied to $X=J(C_f)$ and $Y=J(C_h)$.
\end{proof}

\begin{proof}[Proof of Theorem \ref{mainH}]
In light of Proposition \ref{crucial}(3), we may and will assume that $p=\fchar(K)>0$ and both $J(C_f)$ and $J(C_h)$ are $(n-1)/2$-dimensional supersingular abelian varieties, which are isogenous.  
We may also  assume that $p \ne n$.

The linear disjointness of the splitting fields (property (iii) ) means that
$$\Gal(f/K(\RR_h))=\Gal(f/K)\subset \Perm(\RR_f).$$
In particular, $f(x)$ remains irreducible over $K(\RR_h)$. So,  replacing  $K$ by its overfield $K(\RR_h)$, we may and will assume that $f(x)$ is irreducible  over $K$ and
 $h(x)$ splits into a product of linear factors, i.e.,
the $\Gal(K)$-module $J(C_f)[2]$ is simple (thanks to Proposition \ref{crucial}(1)) while
$$K=K(\RR_h)=K(J(C_h)[2]).$$
The irreducibility of $f(x)$ implies that $[K(\RR_f):K]$ is divisible by $\deg(f)=n$.
Applying Proposition \ref{isogEll} to 
$X=J(C_f), Y=J(C_h)$, we conclude that
$\End^0(J(C_h))$ contains an invertible element, say $c$, of multiplicative order $n$ such that
the $\Q$-subalgebra $\Q[c] \cong\Q(\zeta_n)$.
The supersingularity of $J(C_h)$ implies that $\End^0(J(C_h))$ is isomorphic to the matrix algebra $\mathrm{Mat}_{m}(D_{p,\infty})$ of size $m:=(n-1)/2$ over a definite quaternion algebra $D_{p,\infty}$ such that $D_{p,\infty}$ is ramified precisely at $p$ and $\infty$ \cite[Th. 2d]{Tate}. We obtain that $\mathrm{Mat}_{m}(D_{p,\infty})$  contains a subfield (with the same $1$) that is isomorphic to $\Q(\zeta_n)$. 

Now the desired result follows from Lemma \ref{divisionA}  applied to 
$D=D_{p,\infty}$.
\end{proof}

\section{Proofs of auxiliary results}
\label{crucialP}

\begin{proof}[Proof of Proposition \ref{BourbakiLang}]
Suppose that  the Galois extensions $K(\RR_f)$ and $K(\RR_h)$ of $K$ are {\sl not} linearly disjoint over $K$.
In light of  \cite[A.V.71, Th. 5]{Bourbaki} (see also \cite[Ch. VI, Th. 1.14]{Lang}),
 this means that the field 
$$L:=K(\RR_f)\cap K(\RR_h) \ne K.$$ 

Clearly,
$$K \subset L \subset K(\RR_h).$$

Since the degree $[K(\RR_h):K]=n$ is a prime, $L=K(\RR_h)$, i.e.,
$$K(\RR_f)\cap K(\RR_h)=L=K(\RR_h),$$
which means that $K(\RR_h)\subset K(\RR_f)$. This implies that
$\Z/n\Z\cong \Gal(h)$ is isomorphic to a quotient of $\Gal(f)$, i.e.,
there exists a surjective group homomorphism
$$\psi: G:=\Gal(f) \twoheadrightarrow \Z/n\Z.$$

Suppose that $\Gal(f)$ is {\sl doubly transitive}. We need to arrive to a contradiction.
Let us choose a root $\alpha\in \RR_f$ of $f(x)$  and let $G_{\alpha}$ be the {\sl stabilizer} of $\alpha$ in $G$.
Since $G$ is transitive,  $G_{\alpha}$ is a subgroup of index $n$ in $G$. Since  $n$ is a {\sl prime} and the permutation group 
$$G=\Gal(\RR_f/K)\subset \Perm(\RR_f)$$ 
is isomorphic to
a subgroup of $\mathbf{S}_n$,  the order of $G$ is {\sl not} divisible by $n^2$. This implies that the order of $G_{\alpha}$  is not divisible by $n$,
i.e., is prime to $n$. It follows
that  the order of the image $\psi(G_{\alpha})$ is prime to $n$ as well. Since $\psi(G_{\alpha})$ is a subgroup 
of $\Z/n\Z$,  we get $\psi(G_{\alpha})=\{0\}$, i.e., 
$$G_{\alpha}\subset \ker(\psi) \ \forall \alpha \in \RR_f.$$

 The surjectiveness of $\psi$
implies that the index of $\ker(\psi)$ in $G$ is also $n$. Hence, both  $G_{\alpha}$ and $\ker(\psi)$ have the same order.
Now, the inclusion $G_{\alpha}\subset \ker(\psi)$ implies that 
$$G_{\alpha}=\ker(\psi) \ \forall \alpha \in \RR_f.$$
 
It follows that $ \ker(\psi)$
lies in all the stabilizers $G_{\alpha}$, i.e., $\ker(\psi)$ consists only of the trivial permutation, which sends
every $\alpha\in\RR_f$ to $\alpha$. Hence, the surjective homomorphism $\psi$ is an isomorphism, i.e.,
$\Gal(f) \cong \Z/n\Z$. In particular, the order of $\Gal(f)$ is $n$. However, $\Gal(f)$ is a {\sl doubly transitive}
permutation group of the $n$-element set $\RR_f$. Hence, the order of $\Gal(f)$ is at least $n(n-1)$, which is strictly greater than $n$, because $n \ge 3$.
The obtained contradiction proves the desired result.
\end{proof}

\begin{proof}[Proof of Lemma \ref{simpleX2}]
First,
$$\dim_{\F_2}(X[2])=2g=2\cdot \frac{n-1}{2}=n-1.$$
Second, since $[K(X[2]):K]$ is divisible by the prime $n$,
the Galois group $\tilde{G}_{2,X}=\Gal(K(X[2])/K)$ contains a cyclic subgroup  $H_2$ of order $n$.
Our assumptions on $n$ imply that the faithful representation of $H_2\cong \Z/n\Z$
on the $(n-1)$-dimensional $\F_2$-vector space $X[2]$ is irreducible, see \cite[Lemma 2.8]{Pip}.  Since $H_2$ is a subgroup of $\tilde{G}_{2,X}$,
 the $\tilde{G}_{2,X}$-module $X[2]$ is also simple, which means that the $\Gal(K)$-module $X[2]$ is simple as well.
\end{proof}

\begin{proof}[Proof of Proposition \ref{isogEll}]
There are field inclusions
\begin{equation}
\label{fieldI}
 K \subset K(X[2])\subset K(X[4]).
 \end{equation}
Since $[K(X[2]):K]$ is divisible by $n$, the degree $[K(X[4]):K]$ is also divisible by $n$.
Hence, the Galois group $\Gal(K(X[4)]/K)$ contains a cyclic subgroup  $H$ of order $n$.
Replacing $K$ by   its overfield  $K(X[4])^{H}$ (of $H$-invariants), we may and will assume that $K(X[4])/K$ is a degree $n$ cyclic  extension,
i.e.,  a Galois extension with Galois group $\Gal(K(X[4])/K) \cong \Z/n \Z$.  On the other hand, the field inclusions \eqref{fieldI} give rise to the short exact sequence of Galois groups
\begin{multline}
\{1\} \to  \Gal(K(X[4])/ K(X[2])\to \Gal(K(X[4)])/K)  \\* \to \Gal(K(X[2)])/K)\to \{1\}.\notag
\end{multline}
In particular, $\Gal(K(X[4])/ K(X[2]))$ is a   subgroup of  $\Gal(K(X[4])/K)$ and therefore the order of $\Gal(K(X[4])/ K(X[2]))$
divides the order of $\Gal(K(X[4)])/K)$, which is  {\sl odd} $n$.
 We know 
  that $\Gal(K(X[4])/ K(X[2]))$ is a finite $2$-group (see Remark \ref{silver}). This implies that the order of $\Gal(K(X[4])/ K(X[2]))$ is $1$. In light of the exact sequence,
 we  get that $K(X[4])=K(X[2])$; in particular, 
 $\Gal(K(X[2)])/K)$ is a cyclic group of order $n$ and therefore  $[K(X[2)]):K]=n$.

Recall (Remark \ref{silver}) that $K(Y[4])/K(Y[2])$ is a finite Galois extension, whose degree is a power of $2$. Since 
$K(Y[2])=K$ and $n$ is an odd prime, the fields $K(Y[4])$ and $K(X[4])$ are linearly disjoint over $K$. Replacing $K$ by its overfield $K(Y[4])$, we may and will assume that
$K=K(Y[4])$ and $K(X[4])=K(X[2])$ is still a cyclic  degree $n$ extension of  $K$.  In light of already proven Lemma \ref{simpleX2},
 the $\Gal(K)$-module $X[2]$ is simple.

It follows from the theorem of Silverberg (see Remark \ref{silver} above) applied to $X, Y$ and $X\times Y$  that all the endomorphisms of $Y$ are defined over $K$ and all the homomorphisms from $X$ to $Y$ are defined over $K(X[4])=K(X[2])$.

Suppose that $X$ and $Y$ are isogenous over $\bar{K}$.
Let $\phi: X \to Y$ be an isogeny, which, as we know, is defined over $K(X[4])$. We may assume that $\phi$ has the smallest possible degree; in particular, $\phi$ is {\sl not} divisible by $2$ in $\Hom(X,Y)$.
If $\phi$ is defined over $K$ then it induces a {\sl nonzero} homomorphism of $\Gal(K)$-modules $X[2] \to Y[2]$, which is nonzero, because $\phi$ is {\sl not} divisible by $2$. However, the module $X[2]$ is simple while the module $Y[2]$ is trivial (as a Galois representation), because $K(Y[2])=K$. Hence,  a nonzero homomorphism of these Galois modules does {\sl not} exist and therefore $\phi$ is {\sl not} defined over $K$.   Since both $X$ and $Y$ are defined over $K$,  to each $\sigma \in \Gal(K(X[4])/K) \cong \Z/n Z$ corresponds the isogeny
$\sigma(\phi): X \to Y$, which does {\sl not} coincide with $\phi$ if $\sigma$ is {\sl not} the identity element of $\Gal(K(X[4])/K) $.
Since both $\phi$ and $\sigma(\phi)$ are isogenies from $X$ to $Y$,  there exists precisely one $a_{\sigma}\in \End^0(Y)^{*}$ such that
$$\sigma(\phi)=a_{\sigma} \phi \ \mathrm{in} \ \Hom(X,Y)\otimes \Q.$$
If $\tau \in \Gal(K(X[4])/K) $ then
$$a_{\sigma\tau}(\phi)=(\sigma \tau)\phi =\sigma (\tau(\phi))=\sigma (a_{\tau}\phi)=\sigma (a_{\tau})\sigma(\phi).$$
Since all the endomorphisms of $Y$ are defined over $K$, we have $\sigma (a_{\tau})=a_{\tau}$ and therefore
$$a_{\sigma\tau}(\phi)=a_{\tau}\sigma(\phi)=a_{\tau}(a_{\sigma} \phi)=(a_{\tau}a_{\sigma}) \phi$$
for all $\sigma, \tau \in  \Gal(K(X[4])/K)$. Since  $\Gal(K(X[4])/K) $ is commutative,
$$a_{\tau\sigma}\phi=a_{\sigma\tau}\phi=(a_{\tau}a_{\sigma}) \phi.$$
Since $\phi$ is an isogeny, we get
$$a_{\tau\sigma}=a_{\tau}a_{\sigma} \ \forall \sigma,\tau \in  \Gal(K(X[4])/K).$$
In other words, the map 
$$\Gal(K(X[4])/K) \to \End^0(Y)^{*}, \ \sigma \mapsto a_{\sigma}$$ 
is a group homomorphism, which, as we know, is nontrivial. Take any non-identity element $\sigma \in \Gal(K(X[4])/K)$ and put
$$c=a_{\sigma} \in \End^0(Y)^{*}.$$
Then $c$ is a non-identity element of $\End^0(Y)^{*}$. In addition,   $c^{n}=1$, because the order of $\sigma$ is $n$.
Since $n$ is a prime, the quotient-algebra $\Q[T]/(T^{n}-1)$ of the ring of polynomials $\Q[T]$ is isomorphic to the direct sum
$\Q(\zeta_{n})\oplus\Q$. Hence, the $\Q$-subalgebra $\Q[c]$ of $\End^0(Y)$ is isomorphic to a quotient of $\Q(\zeta_{n})\oplus\Q$;
in particular, it is a {\sl semisimple commutative} $\Q$-(sub)algebra. Notice that
$$2\dim(Y)=n-1<(n-1)+1=\dim_{\Q}\left(\Q(\zeta_{n})\oplus\Q\right).$$
It follows from 
\cite[Ch. II, Prop. 1 on p. 36]{Shimura} that
$$\dim_{\Q}(\Q[c])\le 2\dim(Y)=n-1<\dim_{\Q}(\Q(\zeta_{n})\oplus\Q);$$
in particular, $\Q[c]$ is not isomorphic to $\Q(\zeta_{n})\oplus\Q$.
Since $c$ has odd multiplicative order $n>2$, $\Q[c]$ is not isomorphic to $\Q$.
The only remaining possibility is that there there is an isomorphism of $\Q$-algebras
$\Q[c]\cong \Q(\zeta_{n})$. 
This gives as a $\Q$-algebra embedding
$$\Q(\zeta_{n})\cong \Q[c]\subset \End^0(Y).$$
Since
$$2\dim(Y)=n-1=[\Q(\zeta_{n}):\Q]=\dim_{\Q}(\Q(\zeta_{n})),$$
$Y$ is an abelian variety of CM type over $\bar{K}$ with multiplication by $\Q(\zeta_{n})$.  Since $X$ is $\bar{K}$-isogenous to $Y$,
it is also of CM type with  multiplication by $\Q(\zeta_{n})$.
\end{proof}

\begin{proof}[Proof  of Proposition \ref{crucial}]

Assertion (0) is already proven in Subsection \ref{QRR}. In order to prove (1), notice that the irreducibility of $f(x)$ means the transitivity of permutation group 
$$\Gal(f)=\Gal(f/K)\subset\Perm(\RR_f).$$
Since $n=\#(\RR_f)$ is a prime, the transitivity implies that $n$ divides the order of $\Gal(f)$.
Let 
$$H\subset \Gal(f/K)\subset \Perm(\RR_f)$$
 be a cyclic subgroup of prime order $n$ in $\Gal(f/K)$ and $K_1=K(\RR_f)^H$ the  subfield of $H$-invariants in $K(\RR_f)$. Then
 $$K \subset K_1\subset K(\RR_f), \ K_1(\RR_f)=K(\RR_f);$$
 $$ \Gal(f/K_1)=H \subset \Gal(f/K)\subset \Perm(\RR_f).$$
 Recall that $n$ is a {\sl prime} and $2$ is a {\sl primitive root} $\bmod \ n$, i.e., $2\bmod\ n$ has order $n-1$ in the multiplicative group $(\Z/n\Z)^{*}$.  Since 
 $$\Gal(f/K_1)\cong \Z/n\Z  \ \text{ and } \ n-1=\dim_{\F_2}(Q_{\RR_f}), $$  
 the $\Gal(f/K_1)$-module $Q_{\RR_f}$ is simple, thanks  to \cite[Lemma 2.8]{Pip} (applied to $n=2, p=n, s=n-1$). Since $\Gal(f/K_1)\subset \Gal(f/K)$, the $\Gal(f/K)$-module $Q_{\RR_f}$ is also simple and therefore the $\Gal(K)$-module $Q_{\RR_f}$ is simple as well.  This implies that the $\Gal(K)$-module $J(C_f)[2]$ is also simple, thanks to the arguments at the end of Subsection \ref{QRR}.
 
 Let us prove (2).  We are given that $\Gal(f)$ is {\sl doubly transitive}.  Since $n$ is {\sl odd},  the centralizer $\End_{\Gal(f)}(Q_{\RR_f})$ coincides with $\F_2$,
 thanks to \cite[Satz 4a]{Klemm} (see also \cite[p. 108]{Mori2}).  This implies that the simple $\Gal(f)$-module $Q_{\RR_f}$ is absolutely simple and therefore the $\Gal(K)$-module $J(C_f)[2]$ is also absolutely  simple, thanks to the arguments at the end of Subsection \ref{QRR}.
 
 Let us prove (3). We are given that the fields $K(J(C_f)[2])=K(\RR_f)$ and $K(J(C_h)[2])=K(\RR_h)$ are linearly disjoint over $K$. By already proven assertions (1) and (2) combined with Remark \ref{tildeG}, the $\Gal(K)$-module $J(C_f)[2]$ is absolutely simple while 
 the $\Gal(K)$-module $J(C_h)[2]$ is  simple.  Now the desired result follows from Theorem 2.1 of \cite{ZarhinSh03} (applied to $n=2$ and $X=J(C_f), Y=J(C_h)$).
\end{proof}

\begin{proof}[Proof of Lemma \ref{degSplit}]
 We may view $\Gal(\RR_h)$ as the certain subgroup of $\Perm(\RR_h)$. If the order $N$ of $\Gal(\RR_h)$ is divisible by $n$ then $\Gal(\RR_h)$ contains an element of order $n$ (recall that $n$ is a prime), which is a $n$-cycle. This implies that 
 $\Gal(\RR_h)$
  acts transitively on $\RR_h$, which contradicts the reducibility of $h(x)$. Hence, $N$ is {\sl not} divisible by $n$. On the other hand, $N$ obviously divides $n!$. This implies that
 all prime divisors of $N$ are strictly less than $n$. It remains to notice that 
 $N=[K(\RR_h):K]$.

\end{proof}

\section{Central simple algebras and cyclotomic fields}
\label{cyclicProof}
The aim of this section is to prove  Lemma \ref{divisionA}.
We start with some basic facts related to cyclotomic fields.

\begin{rem}
\label{multOrder}
Suppose that   $p$ and $n$ are {\sl distinct} primes. Let  $\mathcal{O}$ be the ring of integers in the $n$th cyclotomic field $L:=\Q(\zeta_n)$ and $\mathfrak{P}$ a maximal ideal of $\mathcal{O}$ that contains $p  \mathcal{O}$. Let $L_{\mathfrak{P}}$ be the completion of $L$ with respect to the $\mathfrak{P}$-adic topology. 
\begin{itemize}
\item[(i)]
By definition, $L_{\mathfrak{P}}$ contains the fields  $L$ and $\Q_p$ (which coincides with the completion of $\Q$ with respect    to the $\mathfrak{P}$-adic topology).
 Since $n$ is a prime, the field extension  $L/\Q=\Q(\zeta_n)/\Q$ is cyclic
and therefore the field extension $L_{\mathfrak{P}}/\Q_p$ is also cyclic.
\item[(ii)]
 Since the primes $p$ and $n$ are distinct, the field extension $L_{\mathfrak{P}}/\Q_p$ is {\sl unramified} and therefore the degrees of the field extensions
$L_{\mathfrak{P}}/\Q_p$ and $\left(\mathcal{O}/\mathfrak{P}\right)/\left( \Z/p\Z\right)$ do coincide.
\item[(iii)]
In light of \cite[Sect. 1, Lemma 4 and its proof]{ANT}, the residual degree 
$\left[\mathcal{O}/\mathfrak{P}: \Z/p\Z\right]$ coincides with 
the multiplicative order $\mathfrak{f}_p$ of $p \bmod n \in (\Z/n\Z)^{*}$.
This implies that 
\begin{equation}
\label{fp}
\mathfrak{f}_p=\left[L_{\mathfrak{P}}:\Q_p\right].
\end{equation} 
\end{itemize}
\end{rem}
\begin{sect}
We also need to recall basic facts about central simple algebras \cite[Sect. 13.1 and  15.1]{Pierce}.
Let  $\mathcal{K}$ be a field.  If  $\mathcal{A}$ is a central simple algebra over $\mathcal{K}$ then we write $[\mathcal{A}]$
for the similarity class of $\mathcal{A}$ in the {\sl Brauer group} $\mathrm{Br}(\mathcal{K})$.  If $\phi: \mathcal{K} \hookrightarrow \mathcal{L}$
is an {\sl embedding of fields}  then there is the natural group
homomorphism \cite[Sect. 12.5]{Pierce}
\begin{equation}
\phi_{*}: \mathrm{Br}(\mathcal{K}) \to \mathrm{Br}(\mathcal{L}), \ [\mathcal{A}] \to  [\mathcal{A}\otimes_{\mathcal{K}}\mathcal{L}].
\end{equation}

Recall that the correspondences 
$$\mathcal{K} \mapsto \mathrm{Br}(\mathcal{K}), \ \phi \mapsto \phi_{*}$$
define a {\sl functor} from the category of fields to the category of commutative groups  \cite[Sect. 12.5, Proposition {\bf c}]{Pierce}.
In particular, if  $\psi: \mathcal{L} \hookrightarrow \mathcal{F}$ is a {\sl field embedding} then the group homomorphism 
$(\psi\phi)_{*}: \mathrm{Br}(\mathcal{K})  \to \mathrm{Br}(\mathcal{F})$ coincides with the {\sl composition}
\begin{equation}
\label{functor}
\psi_{*} \circ\phi_{*}:  \mathrm{Br}(\mathcal{K})  \to  \mathrm{Br}(\mathcal{L}) \to \mathrm{Br}(\mathcal{F}).
\end{equation} 
 In addition, if $\psi$ is a {\sl field isomorphism}
then $\psi_{*}$ is a {\sl group isomorphism}.

If $\mathcal{L}$ is an {\sl overfield} of $\mathcal{K}$ then we write 
$$i(\mathcal{K},\mathcal{L}): \mathcal{K} \subset \mathcal{L}$$
for the {\sl inclusion map}. We write $\mathbf{B}(\mathcal{L}/\mathcal{K})$ for the {\sl relative Brauer group} of $\mathcal{L}/\mathcal{K}$, i.e., for the {\sl kernel}  of 
the group homomorphism
$$i(\mathcal{K},\mathcal{L})_{*}: \mathrm{Br}(\mathcal{K})  \to  \mathrm{Br}(\mathcal{L}).$$ 
\cite[Sect. 13.2]{Pierce}. If $\mathcal{F}$ is an {\sl overfield} of $\mathcal{L}$ then,  by functoriality \eqref{functor},
\begin{equation}
\label{compBrauer}
i(\mathcal{K},\mathcal{F})_{*}=i(\mathcal{L},\mathcal{F})_{*}\circ i(\mathcal{K},\mathcal{L})_{*}.
\end{equation}
If $\mathcal{L}_1$ and $\mathcal{L}_2$ are {\sl overfields} of $\mathcal{K}$ and there is a field
isomorphism $\psi: \mathcal{L}_1 \to \mathcal{L}_2$, whose restriction to $\mathcal{K}$  coincides with the
{\sl identity map} then
$$ i(\mathcal{K},\mathcal{L}_2)=\psi\circ  i(\mathcal{K},\mathcal{L}_1)$$ and therefore
$$ i(\mathcal{K},\mathcal{L}_2)_{*}=\psi_{*}\circ  i(\mathcal{K},\mathcal{L}_1)_{*}.$$

Since $\psi_{*}: \mathrm{Br}(\mathcal{L}_1) \to \mathrm{Br}(\mathcal{L}_2)$ is a group isomorphism,
the kernels of $i(\mathcal{K},\mathcal{L}_2)_{*}$ and $i(\mathcal{K},\mathcal{L}_1)_{*}$ do coincide, i.e.
\begin{equation}
\label{relativeB}
\mathbf{B}(\mathcal{L}_1/\mathcal{K})=\mathbf{B}(\mathcal{L}_2/\mathcal{K}).
\end{equation}

\begin{defn}
Let $\mathcal{A}$ be a central simple algebra over $\mathcal{K}$ of finite dimension $d^2$ where $d$ is a positive integer.
A $\mathcal{K}$-subalgebra $E$ of $\mathcal{A}$ with the same $1$ that is a subfield is called a {\sl strictly maximal subfield} if $[E:\mathcal{K}]=d$.
(See \cite[Sect. 13.1]{Pierce}.) 
\end{defn}

\begin{rem}
\label{StrictSplit}
It is known \cite[Sect. 13.3]{Pierce} that if $E$ is a  strictly maximal subfield of a central simple $\mathcal{K}$-algebra $\mathcal{A}$ 
then $E$ {\sl splits} $\mathcal{A}$, i.e., the $E$-algebra $\mathcal{A}\otimes_{\mathcal{K}}E$ is isomorphic to a matrix algebra over $E$.
In other words, $[\mathcal{A}]\in \mathbf{B}(E/\mathcal{K})$.
\end{rem}
\end{sect}

\begin{proof}[Proof of Lemma \ref{divisionA}]
 
We keep the notation and conventions of Remark \ref{multOrder}.
Let us put $\mathcal{A}:=\Mat_m(D)$. Then 
$\mathcal{A}$ is a central simple $\Q$-algebra of dimension 
$$4 \cdot m^2=(2m)^2=(n-1)^2$$
and therefore $E$ is a {\sl strictly maximal subfield}  of the central simple $\Q$-algebra $\mathcal{A}$, because
$$[E:\Q]=[\Q(\zeta_n):\Q]=n-1=\sqrt{\dim_{\Q}(\mathcal{A})}.$$

This implies that
$[\mathcal{A}] \in \mathbf{B}(E/\Q)$. In light of \eqref{relativeB},
$[\mathcal{A}] \in \mathbf{B}(L/\Q),$
 because the fields  $E$ and $L$ are isomorphic (recall that $L=\Q(\zeta_n)$).  In other words, 
\begin{equation}
\label{brauer0}
i(\Q,L)_{*}[\mathcal{A}]=0.
\end{equation}

Now let us consider the central simple $\Q_p$-algebra
$$\mathcal{A}_p:=\mathcal{A}\otimes_{\Q}\Q_p=\Mat_m(D)\otimes_{\Q}\Q_p=\Mat_m(D_p)$$
where
$D_p:=D\otimes_{\Q}\Q_p$ is a quaternion (division) algebra over $\Q_p$. This implies that
$\mathcal{A}_p$ is {\sl not} isomorphic to a matrix algebra over $\Q_p$ but $\mathcal{A}_p \otimes_{\Q_p}\mathcal{A}_p$ is isomorphic to a matrix algebra over $\Q_p$.
In other words, 
$$[\mathcal{A}_p]=i(\Q,\Q_p)_{*}[\mathcal{A}]$$
 is an element of order $2$ in $\mathrm{Br}(\Q_p)$.  Let us prove that
\begin{equation}
\label{brauerC}
[\mathcal{A}_p]\in \mathbf{B}(L_{\mathfrak{P}}/\Q_p)
\end{equation}
where the field $L_{\mathfrak{P}}$ is  defined in Remark \ref{multOrder}. Indeed,  recall that
$L_{\mathfrak{P}}$ contains both fields $\Q_p$ and $L=\Q(\zeta_n)$, whose intersection contains $\Q$.
By functoriality of Brauer groups \eqref{compBrauer},
\begin{equation}
\label{functorB}
i(\Q_p, L_{\mathfrak{P}})_{*}\circ i(\Q,\Q_p)_{*}=i(\Q, L_{\mathfrak{P}})_{*}=i(L, L_{\mathfrak{P}})_{*}  \circ i(\Q, L)_{*}.
\end{equation}
This implies that 
\begin{equation}
\label{brauer2}
i(\Q_p, L_{\mathfrak{P}})_{*}[\mathcal{A}_p]=i(\Q_p, L_{\mathfrak{P}})_{*}\circ \left(i(\Q,\Q_p)_{*} [\mathcal{A}]\right)=
i(L, L_{\mathfrak{P}})_{*}\circ \left(i(\Q,L)_{*}[\mathcal{A}]\right).
\end{equation}

In light of \eqref{brauer0},  $i(\Q,L)[\mathcal{A}]=0$. In light  of \eqref{brauer2},  
$i(\Q_p, L_{\mathfrak{P}})[\mathcal{A}_p]=0$, which proves \eqref{brauerC}. 

It follows that $\mathbf{B}(L_{\mathfrak{P}}/\Q_p)$ contains an element of order $2$, namely $[\mathcal{A}_p]$.
Since $L_{\mathfrak{P}}/\Q_p$ is a {\sl cyclic} field extension (see Remark \ref{multOrder}),
the group 
$\mathbf{B}(L_{\mathfrak{P}}/\Q_p)$
 is isomorphic to the quotient
$\Q_p^{*}/\mathrm{N}_{L_{\mathfrak{P}}/\Q_p}(L^{*})$ where
$$\mathrm{N}_{L_{\mathfrak{P}}/\Q_p}: L_{\mathfrak{P}}^{*} \to \Q_p^{*}$$
is the {\sl norm map} attached to  $L/\Q_p$ and $\mathrm{N}_{L_{\mathfrak{P}}/\Q_p}(L_{\mathfrak{P}}^{*})$ is its image in $\Q_p^{*}$
\cite[Sect. 15.1, Proposition {\bf b}]{Pierce}.
On the other hand,  according to the {\sl fundamental equality} in local class field theory  \cite[Sect. 7.1, p. 98]{Iwasawa}, the index
$\left[\Q_p^{*}:\mathrm{N}_{L_{\mathfrak{P}}/\Q_p}(L_{\mathfrak{P}}^{*})\right]$
coincides with the degree $[L_{\mathfrak{P}}:\Q_p]$. 
\begin{footnote}
{Actually, we need only a special ``elementary'' case of the fundamental equality that deals
with {\sl unramified} extensions  and is discussed in detail  in \cite[Ch. V,  \S 2]{SerreCorps}.}
\end{footnote}
 This means that the quotient $\Q_p^{*}/\mathrm{N}_{L_{\mathfrak{P}}/\Q_p}(L_{\mathfrak{P}}^{*})$
is a finite group of order  $[L_{\mathfrak{P}}:\Q_p]$. Since $\Q_p^{*}/\mathrm{N}_{L_{\mathfrak{P}}/\Q_p}(L_{\mathfrak{P}}^{*})$ contains an element of order 2,
the degree $[L_{\mathfrak{P}}:\Q_p]$ is an {\sl even integer}.
 On the other hand,  \eqref{fp} tells us that 
$[L_{\mathfrak{P}}:\Q_p]$ 
  coincides with the multiplicative order 
  $\mathfrak{f}_p$  
  of the residue $p \bmod n \in (\Z/n\Z)^{*}$. Hence, $\mathfrak{f}_p$ is even. This ends the proof.

\end{proof}

\end{document}